\documentclass[a4paper,reqno,11pt]{amsart}
\usepackage{amsmath, graphicx}
\usepackage{amstext,amsfonts,amsbsy,eucal,amssymb}

\numberwithin{equation}{section}

\newtheorem{theorem}{Theorem}[section]
\newtheorem{lemma}{Lemma}[section]
\newtheorem{proposition}{Proposition}[section]
\newtheorem{conjecture}{Conjecture}[section]
\newtheorem{claim}{Claim}[section]

\theoremstyle{definition}
\newtheorem{definition}{Definition}[section]

\newtheorem{remark}{Remark}[section]

\newcommand{\R}{\operatorname{\mathbb{R}}}
\newcommand{\Z}{\operatorname{\mathbb{Z}}}
\newcommand{\C}{\operatorname{\mathbb{C}}}
\newcommand{\Q}{\operatorname{\mathbb{Q}}}
\newcommand{\ve}{\operatorname{\varepsilon}}

\newcommand{\Jac}{\operatorname{\mathrm{Jac}}}
\newcommand{\Det}{\operatorname{\mathrm{Det}}}

\newcommand{\mT}{\operatorname{\mathcal{T}}}
\newcommand{\mF}{\operatorname{\mathcal{F}}}
\newcommand{\bZ}{\operatorname{\mathbf{Z}}}

\newcommand{\Val}{\operatorname{\mathrm{Val}}}
\newcommand{\val}{\operatorname{\mathrm{val}}}
\newcommand{\ee}{\bf e}

\begin{document}

\title{Tropical spectral curves, Fay's trisecant identity, and generalized ultradiscrete Toda lattice}

\dedicatory{ 
Dedicated to Professor Tetsuji Miwa on his 60th birthday
}

\author{Rei Inoue}
\address{Faculty of Pharmaceutical Sciences, 
Suzuka University of Medical Science
\newline\phantom{iii}3500-3 Minami-tamagaki, Suzuka, Mie, 513-8670, Japan}
\email{reiiy@suzuka-u.ac.jp}

\author{Shinsuke Iwao}
\address{Graduate School of Mathematical Sciences,
The University of Tokyo
\newline\phantom{iii}3-8-1 Komaba Meguro-ku, Tokyo 153-8914, Japan}
\email{iwao@ms.u-tokyo.ac.jp}

\begin{abstract}
We study the generalized ultradiscrete periodic Toda lattice 
$\mT(M,N)$ which has tropical spectral curve.
We introduce a tropical analogue of Fay's trisecant identity, 
and apply it to construct a general solution to $\mT(M,N)$. 
\end{abstract}


\maketitle

\section{Introduction}

The {\em ultradiscrete} periodic Toda lattice is 
an integrable system described by 
a piecewise-linear map \cite{KimijimaTokihiro02}.
Recently, its algebro geometrical aspect is clarified
\cite{IT08,IT09,IT09b}
by applying the tropical geometry, a combinatorial algebraic geometry
rapidly developed during this decade \cite{EKL06,IMS-Book, SpeyerSturm04}.
This system has tropical spectral curves, and what proved are 
that its general isolevel set is isomorphic to the tropical
Jacobian of the tropical hyperelliptic curve, 
and that its general solution is written in terms of 
the tropical Riemann's theta function.
The key to the solution is the tropical analogue of Fay's trisecant identity
for a special family of hyperelliptic curves \cite{IT09}.

On the other hand, 
there exists a generalization of {\em discrete} periodic Toda lattice $T(M,N)$,
where $M$ (resp. $N$) is a positive integer
which denotes the level of generalization (resp. the periodicity)  
of the system. 
The $M=1$ case, $T(1,N)$, is the original discrete Toda lattice
of $N$-periodicity.
When $\gcd(M,N)=1$, $T(M,N)$ reduces to a special case 
of the integrable multidiagonal Lax matrix \cite{vanMoerMum79},
and the general solution to $T(M,N)$ is recently constructed \cite{Iwao09}.

The aim of this paper is twofold: 
the first one is to introduce the tropical analogue of Fay's trisecant identity
not only for hyperelliptic but also for more general tropical curves.
The second one is to study 
the generalization of {\em ultradiscrete} periodic Toda lattice $\mT(M,N)$
by applying the tropical Fay's trisecant identity,
as a continuation of the study on $\mT(1,N)$ \cite{IT08,IT09,IT09b}.

This paper is organized as follows:
in \S 2 we review some notion of tropical geometry 
\cite{Mikha05,MikhaZhar06,Iwao08}, and 
introduce the tropical analogue of Fay's trisecant identity
(Theorem \ref{tropicalFay}) by applying the correspondence of integrations
over complex and tropical curves \cite{Iwao08}.
In \S 3 we introduce the generalization of the discrete periodic Toda lattice 
$T(M,N)$ and its ultradiscretization $\mT(M,N)$.
We reconsider the integrability of $T(M,N)$ 
(Proposition \ref{prop:integrability}).
In \S 4 we demonstrate the general solution
to $\mT(3,2)$, and give conjectures on $\mT(M,N)$ 
(Conjectures \ref{conj:1} and \ref{conj:2}).

In closing the introduction, we make a brief remark on  
the interesting close relation between 
the ultradiscrete periodic Toda lattice and 
the {\em periodic box and ball system} (pBBS) \cite{KimijimaTokihiro02},
which is generalized to that between $\mT(M,N)$ and 
the pBBS of $M$ kinds of balls
\cite{NagaiTokihiroSatsuma99,Iwao09b}.
When $M=1$, the relation is explained at the level of 
tropical Jacobian \cite{IT08}.
We expect that our conjectures on $\mT(M,N)$ also account for
the tropical geometrical aspects of the recent results 
\cite{KunibaTakagi09} on the pBBS of $M$ kinds of balls.

\section{Tropical curves and Riemann's theta function}
\subsection{Good tropicalization of algebraic curves}

Let $K$ be a subfield of $\C$ and 
$K_{\ve}$ be the field of convergent Puiseux series 
in ${\bf e} := {\rm e}^{-1/\ve}$ over $K$.
Let $\val:K_{\ve}\to \Q\cup \{\infty\}$ be the natural valuation with respected to 
$\ee$.
Any polynomial $f_{\ve}$ in $K_{\ve}[x,y]$ is expressed uniquely as 
\[\textstyle
f_{\ve}=\sum_{w=(w_1,w_2)\in \Z^2}{a_w(\ve)x^{w_1}y^{w_2}},\qquad
a_w(\ve)\in K_{\ve}.
\]
Define the tropical polynomial $\Val(X,Y;f_{\ve})$ associated with $f_{\ve}$
by the formula:
$\textstyle
\Val(X,Y;f_{\ve}):=\min_{w\in\Z^2}{[\val(a_w)+w_1X+w_2Y]}.
$
We call $\Val(X,Y;f_{\ve})$ the {\em tropicalization} of $f_{\ve}$.

We take $f_{\ve} \in K_{\ve}[x,y]$. 
Let $C^0(f_{\ve})$ be the affine algebraic
curve over $K_{\ve}$ defined by $f_{\ve}=0$.
We write $C(f_{\ve})$ for 
the complete curve over $K_{\ve}$
such that
$C(f_{\ve})$ contains $C^0(f_{\ve})$ as a dense open subset,
and that $C(f_{\ve}) \setminus C^0(f_{\ve})$ consists of non-singular points.  
The \textit{tropical curve} $TV(f_{\ve})$ is a subset of $\R^2$
defined by:
\[
TV(f_{\ve})=\left\{(X,Y)\in\R^2\,\left\vert\,
\begin{array}{l}
\mbox{the function $\Val(X,Y;f_{\ve})$}\\
\mbox{ is not smooth at $(X,Y)$ }
\end{array}
\right\}\right..
\]
Denote $\Lambda(X,Y;f_{\ve}):=\{w\in \Z^2\,\vert\, \Val(X,Y;f_{\ve})
=\val(a_w)+w_1X+w_2Y
\}$.
The definition of the tropical curve can be put into:
\[
TV(f_{\ve})=\{(X,Y)\in\R^2\,\vert\,
\sharp\Lambda(X,Y;f_{\ve})\geq 2
\}.
\]
For $P=(X,Y)\in \R^2$, we define $f_{\ve}^P:=\sum_{w\in\Lambda(X,Y;f_{\ve})}{a_w x^{w_1}y^{w_2}}$.


To make use of the results of tropical geometry for
real/complex analysis, 
we introduce the following condition as a criterion 
of genericness of tropical curves.

\begin{definition}\label{def2.1}
We say $C(f_{\ve})$ has a \textit{good tropicalization} if:
\begin{enumerate}
\def\labelenumi{(\theenumi)}
\item $C(f_{\ve})$ is an irreducible reduced non-singular curve over $K_{\ve}$,
\item $f^P_{\ve}=0$ defines an affine reduced 
non-singular curve in $(K_{\ve}^\times)^2$ for all 
$P\in TV(f_{\ve})$ (maybe reducible).
\end{enumerate}
\end{definition}

\begin{remark}
  The notion of a {\em good tropicalization} was first introduced in
  \cite[Section 4.3]{Iwao08}.
  The above definition gives essentially the same notion. 
\end{remark}

\subsection{Smoothness of tropical curves}

For the tropical curve $\Gamma := TV(f_{\ve})$,
we define the set of vertices $V(\Gamma)$:
\begin{align*}
  V(\Gamma) = \{(X,Y) \in \Gamma ~|~ \sharp\Lambda(X,Y;f_{\ve})\geq 3 \}.
\end{align*}
We call each disjointed element of $\Gamma \setminus V(\Gamma)$ 
an edge of $\Gamma$.
For an edge $e$, we have the {\it primitive tangent vector} 
$\xi_e = (n,m) \in \Z^2$ as $\gcd(n,m) = 1$.
Note that the vector $\xi_e$ is uniquely determined up to sign.

\begin{definition}\cite[\S 2.5]{Mikha05} 
  The tropical curve $\Gamma$ is {\it smooth} if:
  \begin{enumerate}
  \def\labelenumi{(\theenumi)}
  \item All the vertices are trivalent, {\it i.e.} 
  $\sharp\Lambda(X,Y;f_{\ve})=3$ for all $(X,Y) \in V(\Gamma)$.

  \item For each trivalent vertex $v \in V(\Gamma)$, 
  let $\xi_1, \xi_2$ and $\xi_3$ be the primitive tangent vectors of 
  the three outgoing edges from $v$. Then we have 
  $\xi_1 + \xi_2 + \xi_3 = 0$ and $|\xi_i \wedge \xi_j| = 1$ 
  for $i \neq j \in \{1,2,3\}$.
  \end{enumerate}
\end{definition}

When $\Gamma$ is smooth, the genus of $\Gamma$ is $\dim H_1(\Gamma,\Z)$.

\subsection{Tropical Riemann's theta function}

For an integer $g \in \Z_{>0}$,
a positive definite symmetric matrix $B \in M_g(\R)$ and 
$\beta \in \R^g$ we define a function on $\R^g$ as
$$ 
  q_\beta({\bf m},{\bf Z}) = \frac12 {\bf m}B{\bf m}^{\bot}+
  {\bf m}({\bf Z}+\beta B)^{\bot}
  \qquad ({\bf Z} \in \R^g, ~ {\bf m} \in \Z^g).
$$
The tropical Riemann's theta function $\Theta(\bZ;B)$ and its
generalization $\Theta[\beta](\bZ;B)$ are given by \cite{MikhaZhar06,IT09}
\begin{align*}
  &\Theta({\bf Z};B)=\min_{{\bf m}\in \Z^g} q_0({\bf m},{\bf Z}),
  \\
  &\Theta[\beta]({\bf Z};B) = \frac12 \beta B \beta^{\bot}+
   \beta {\bf Z}^{\bot}+
   \min_{{\bf m} \in \Z^g} q_\beta({\bf m},{\bf Z}).
\end{align*}
Note that $\Theta[0](\bZ;B) = \Theta(\bZ;B)$.
The function $\Theta[\beta](\bZ;B)$ satisfies the quasi-periodicity:
\begin{align*}
  \Theta[\beta]({\bf Z}+K{\bf l})
  =
  -\frac{1}{2} {\bf l}K{\bf l}^{\bot} -{\bf l}{\bf Z}^{\bot}
  +\Theta[\beta]({\bf Z})  \qquad ({\bf l}\in \Z^g). 
\end{align*}
We also write $\Theta({\bf Z})$ and $\Theta[\beta]({\bf Z})$
for $\Theta({\bf Z};B)$ and $\Theta[\beta]({\bf Z};B)$
without confusion.
We write ${\bf n} = \arg_{{\bf m} \in \Z^g} q_\beta({\bf m},{\bf Z})$ 
when $\min_{{\bf m} \in \Z^g} q_\beta({\bf m},{\bf Z}) 
= q_\beta({\bf n},{\bf Z})$.

\subsection{Tropical analogue of Fay's trisecant identity}

For each $\bar{\ve} \in \R_{>0}$,
we write $C(f_{\bar{\ve}})$ for the base change of $C(f_{\ve})$
to $\C$ via a map
$\iota : ~K_{\ve} \to \C$ given by $\ve \mapsto \bar{\ve}$.

\begin{theorem}\label{Trop-period}\cite[Theorem 4.3.1]{Iwao08}
Assume $C(f_{\ve})$ has a good tropicalization and
$C(f_{\bar{\ve}})$ is non-singular.
Let $B_{\bar{\ve}}$ and $B_T$ be the period matrices 
for $C(f_{\bar{\ve}})$ and $TV(f_{\ve})$ respectively.
Then we have the relation
$$
  \frac{2\pi\bar{\ve}}{\sqrt{-1}} B_{\bar{\ve}} 
  \sim  B_T \quad (\bar{\ve} \to 0).
$$
(It follows from the assumption that 
the genus of $C(f_{\ve})$ and $C(f_{\bar{\ve}})$ coincide.) 
\end{theorem}

A nice application of this theorem is to give
the tropical analogue of Fay's trisecant identity.
For the algebraic curve $C(f_{\ve})$ of Theorem \ref{Trop-period}, 
we have the following:  
 
\begin{theorem}\label{tropicalFay}
We continue the hypothesis and notation in Theorem \ref{Trop-period},
and assume $TV(f_{\ve})$ is smooth.
Let $g$ be the genus of $C(f_{\ve})$ and 
$(\alpha,\beta) \in \frac12 \Z^{2g}$ be a non-singular 
odd theta characteristic for $\Jac(C(f_{\ve}))$.
For $P_1, P_2, P_3, P_4$ 
on the universal covering space of $TV(f_{\ve})$, 
we define the sign $s_i\in \{\pm 1\}$ $(i=1,2,3)$ as 
$s_i=(-1)^{k_i}$,
where
\begin{align*}
\begin{split}
k_1&=2\alpha\cdot 
\left(\arg_{{\bf m} \in \Z^g} q_\beta({\bf m},\int_{P_3}^{P_2})+
\arg_{{\bf m} \in \Z^g} q_\beta({\bf m},\int_{P_1}^{P_4})\right),\\
k_2&=2\alpha\cdot 
\left(\arg_{{\bf m} \in \Z^g} q_\beta({\bf m},\int_{P_3}^{P_1})+
\arg_{{\bf m} \in \Z^g} q_\beta({\bf m},\int_{P_4}^{P_2})\right),\\
k_3&=1+2\alpha\cdot 
\left(\arg_{{\bf m} \in \Z^g} q_\beta({\bf m},\int_{P_3}^{P_4})+
\arg_{{\bf m} \in \Z^g} q_\beta({\bf m},\int_{P_1}^{P_2})\right).
\end{split}
\end{align*}
Set the functions $F_1(\bZ), F_2(\bZ), F_3(\bZ)$ of
$\bZ \in \R^g$ as
\begin{align*}%
\begin{split}
F_1(\bZ)&=\Theta({\bf Z}+\int_{P_1}^{P_3})+\Theta({\bf Z}+\int_{P_2}^{P_4})
 +\Theta[\beta](\int_{P_3}^{P_2})+\Theta[\beta](\int_{P_1}^{P_4}),\\
F_2(\bZ)&=\Theta({\bf Z}+\int_{P_2}^{P_3})+\Theta({\bf Z}+\int_{P_1}^{P_4})
 +\Theta[\beta](\int_{P_3}^{P_1})+\Theta[\beta](\int_{P_4}^{P_2}),\\
F_3(\bZ)&=\Theta({\bf Z}+\int_{P_1+P_2}^{P_3+P_4})+\Theta({\bf Z})
+\Theta[\beta](\int_{P_4}^{P_3})+\Theta[\beta](\int_{P_1}^{P_2}).
\end{split}
\end{align*}
Then, the formula
\begin{align}\label{trop-fay}
  F_i(\bZ)=\min[F_{i+1}(\bZ),F_{i+2}(\bZ)]
\end{align}  
holds if $s_i=\pm 1,s_{i+1}=s_{i+2}=\mp 1$ for $i \in \Z / 3 \Z$.
\end{theorem}

This theorem generalizes \cite[Theorem 2.4]{IT08}, 
where $C(f_{\ve})$ is a special hyperelliptic curve.
We introduce the following lemma for later convenience:

\begin{lemma}\label{lemma:theta-ch}\cite[Proposition 2.1]{IT08} 
  Let $C$ be a hyperelliptic curve of genus $g$ and take
  $\beta = (\beta_j)_j \in \frac{1}{2}\Z^g$.
  Set $\alpha = (\alpha_j)_j \in \frac{1}{2}\Z^g$ as
  $$ 
    \alpha_j = - \frac{1}{2} \delta_{j,i-1} + \frac{1}{2} \delta_{j,i},
  $$
  where $i$ is defined by the condition
  $\beta_j=0 ~(1\leq j \leq i-1)$ and $\beta_i\neq 0$ mod $\Z$.
  Then $(\alpha,\beta)$ is a non-singular odd theta characteristic 
  for $\Jac(C)$. 
\end{lemma}

\subsection{Tropical Jacobian}

When the positive definite symmetric  
matrix $B \in M_g(\R)$ is the period matrix of 
a smooth tropical curve $\Gamma$, 
the $g$-dimensional real torus $J(\Gamma)$ defined by
$$
  J(\Gamma) := \R^g / \Z^g B
$$
is called the tropical Jacobian \cite{MikhaZhar06} of $\Gamma$.

\section{Discrete and ultradiscrete generalized Toda lattice}
\subsection{Generalized discrete periodic Toda lattice $T(M,N)$}

Fix $M, N \in \Z_{>0}$.
Let $T(M,N)$ be the generalization of discrete periodic Toda lattice 
defined by the difference equations \cite{Iwao09,NagaiTokihiroSatsuma99}
\begin{align}
  \label{slN-dpToda}
  \begin{split}
  &I_n^{t+1} + V_{n-1}^{t+\frac{1}{M}} 
  = I_n^t + V_n^t, 
  \\
  &V_n^{t+\frac{1}{M}} I_n^{t+1} = I_{n+1}^t V_n^t,
  \end{split}
  \quad (n \in \Z /N \Z, ~t \in \Z/M),
\end{align}
on the phase space $T$:
\begin{align}\label{spaceT}
  \begin{split}
  &\Big\{(I_{n}^t, I_{n}^{t+\frac{1}{M}}, \ldots, I_n^{t+\frac{M-1}{M}}, V_n^t)_{n=1,  \ldots,N} \in \C^{(M+1)N} 
  \\ 
  & \qquad ~\Big|~ 
  \prod_{n=1}^N V_n^t, ~\prod_{n=1}^N I_n^{t+\frac{k}{M}} ~(k=0,\ldots,M-1) 
  \text{ are distinct } \Big\}.
  \end{split}
\end{align}
Eq. \eqref{slN-dpToda} is written in the Lax form given by
$$
  L^{t+1}(y) M^t(y) = M^t(y) L^t(y), 
$$
where
\begin{align}\label{MNLax} 
  &L^t(y) = M^t(y) R^{t+\frac{M-1}{M}}(y) \cdots R^{t+\frac{1}{M}}(y) R^t(y),
  \\
  \nonumber
  &R^t(y) = 
    \begin{pmatrix} 
      I_2^t & 1 & & &  \\
       & I_3^t & 1 & & \\
      & & \ddots & \ddots & \\
      & &  & I_{N}^t & 1 \\
      y& & & & I_1^t \\
    \end{pmatrix},
  \quad 
  M^t(y) = 
    \begin{pmatrix} 
      1 &  &  & & \frac{V_1^t}{y} \\
      V_2^t & 1   & & \\
      & V_3^t & 1 \\
      & & \ddots & \ddots  & \\
      & & & V_{N}^t & 1 \\
    \end{pmatrix}.
\end{align}
The Lax form ensures that the characteristic polynomial 
$\Det(L^t(y) - x \mathbb{I}_N)$ of the Lax matrix $L^t(y)$
is independent of $t$, namely, the coefficients of 
$\Det(L^t(y) - x \mathbb{I}_N)$ are integrals of motion of $T(M,N)$.

Assume $\gcd(M,N)=1$ and 
set $d_j = [\frac{(M+1-j)N}{M}] ~(j=1,\ldots,M)$.
We consider three spaces for $T(M,N)$: 
the phase space $T$ \eqref{spaceT}, 
the coordinate space $L$ for the Lax matrix $L^t(y)$ \eqref{MNLax},
and the space $F$ of the spectral curves.
The two spaces $L$ and $F$ are given by
\begin{align}
  &L = \{(a_{i,j}^t,b_i^t)_{i=1,\ldots,M, ~j=1,\dots,N} \in \C^{(M+1)N} \}, 
  \\
  \begin{split}\label{slN-dpF}
  &F = \Big\{ y^{M+1} + f_M(x) y^{M} + \cdots + f_1(x) y + f_{0} 
  \in \C[x,y] ~\Big|~
  \\
  & \qquad \qquad 
  \deg_x f_j(x) \leq d_j ~(j=1,\ldots, M),
  ~-f_1(x) \text{ is monic}
  \Big\},
  \end{split}
\end{align}
where each element in $L$ corresponds to the matrix:
\begin{align*}
  L^t(y) = 
      \begin{pmatrix} 
      a_{1,1}^t & a_{2,2}^t & \cdots & a_{M,M}^t & 1 &  & & \frac{b_N^t}{y}\\
      b_1^t & a_{1,2}^t & a_{2,3}^t & \cdots & a_{M,M+1}^t& 1 & \\
      & \ddots & \ddots & \ddots & & \ddots & \ddots \\[2mm]
      & & \ddots & & & & a_{M,N-1}^t & 1 \\
      y &  & & \ddots& & & a_{M-1,N-1}^t & a_{M,N}^t \\
      y a^t_{M,1} & y &  & & & & & a_{M-1,N}^t \\
      \vdots & \ddots & \ddots &  & & \ddots & \ddots & \vdots\\[2mm]
      y a^t_{2,1} & \cdots & y a^t_{M,M-1} & y & & & b_{N-1}^t & a_{1,N}^t \\
    \end{pmatrix}.
\end{align*}

Define two maps $\psi : T \to L$ and $\phi : L \to F$ given by
\begin{align*}
  &\psi((I_{n}^t, I_{n}^{t+\frac{1}{M}}, \ldots, I_n^{t+\frac{M-1}{M}}, V_n^t)_{n=1,  \ldots,N}) = L^t(y) 
  \\
  &\phi(L^t(y)) = (-1)^{N+1} y \Det(L^t(y) - x \mathbb{I}_N),
\end{align*}
Via the map $\psi$ (resp. $\phi \circ \psi$), we can regard 
$F$ as a set of polynomial functions on $L$ (resp. $T$). 
We write $n_F$ for the number of the polynomial functions in $F$,
which is $\displaystyle{n_F = \frac{1}{2}(M+1)(N+1)}$.

\begin{proposition}\label{prop:integrability}
  The $n_F$ functions in $F$ are functionally independent in $\C[T]$.
\end{proposition}

To prove this proposition we use the following:
\begin{lemma}\label{lemma:psi}
  Define
  $$
  I^{t+\frac{k}{M}} = \prod_{n=1}^N 
     I_n^{t+\frac{k}{M}}, 
  \quad
  V^t = \prod_{n=1}^N V_n^t ~~(t \in \Z, ~ k=0,\ldots,M-1).
  $$
  The Jacobian of $\psi$ does not vanish iff
  $I^{t+\frac{k}{M}} \neq I^{t+\frac{j}{M}}$ for $0 \leq k < j \leq M-1$
  and $I^{t+\frac{k}{M}} \neq V^t$ for $0 \leq k \leq M-1$.  
\end{lemma}
\begin{proof}
  Since the dimensions of $T$ and $L$ are same,
  the Jacobian matrix of $\psi$ becomes an $(M+1)N$ by $(M+1)N$ matrix.
  By using elementary transformation, one sees that the Jacobian matrix is 
  block diagonalized into $M+1$ matrices of $N$ by $N$,
  and the Jacobian is factorized as
  $$
    \pm \Det B \cdot \prod_{k=1}^{M-1} \Det A^{(k)},  
  $$
  where $A^{(k)}$ and $B$ are 
\begin{align*}
  &A^{(k)} = P(I^{t+\frac{k}{M}}, I^{t+\frac{k-1}{M}}) 
             P(I^{t+\frac{k+1}{M}}, I^{t+\frac{k-1}{M}}) \cdots
             P(I^{t+\frac{M-1}{M}}, I^{t+\frac{k-1}{M}}) 
  \\
  &\hspace{8cm} (k=1,\ldots,M-1), 
  \\
  &B = P(I^{t+\frac{M-1}{M}},V^t) P(I^{t+\frac{M-2}{M}},V^t) \cdots
       P(I^{t+\frac{1}{M}},V^t) P(I^{t},V^t),     
  \\[1mm]
  &P(J,K) = \begin{pmatrix}
              J_1 & & & & -K_N \\[1mm]
              -K_1 & J_2 & & & \\
              & \ddots & \ddots & & \\
              & & \ddots & \ddots & \\[1mm]
              & & & -K_{N-1} & J_N
            \end{pmatrix} \in M_N(\C).
\end{align*}
  Since 
  $\Det P(J,K) = \prod_{n=1}^N J_n - \prod_{n=1}^N K_n$,
  we obtain
  $$
  \Det B \cdot \prod_{k=1}^{M-1} \Det A^{(k)}
  = \prod_{0 \leq k \leq M-1} (I^{t+\frac{k}{M}}-V^t) \cdot
    \prod_{0 \leq j < k \leq M-1} (I^{t+\frac{k}{M}}-I^{t+\frac{j}{M}}).
  $$  
  Thus the claim follows.
\end{proof}

\begin{remark}
  The above lemma is true for $\gcd(M,N) > 1$, too.
\end{remark}

\begin{proof}(Proposition \ref{prop:integrability}) \\
Take a generic $f \in F$ such that the algebraic curve $C_f$ given by
$f=0$ is smooth. The genus $g$ of $C_f$ is $\frac{1}{2} (N-1)(M+1)$,
and we have $\dim L = n_F + g$.
Due to the result by Mumford and van Moerbeke
\cite[Theorem 1]{vanMoerMum79},
the isolevel set $\phi^{-1}(f)$ is isomorphic to the affine part of   
the Jacobian variety $\mathrm{Jac}(C_f)$ of $C_f$,
which denotes $\dim_{\C} \phi^{-1}(f) = g$.
Thus $F$ has to be a set of $n_F$ functionally 
independent polynomials in $\C[L]$.
Then the claim follows from Lemma \ref{lemma:psi}.

\end{proof}

\subsection{Generalized ultradiscrete periodic Toda lattice $\mT(M,N)$}

We consider the difference equation 
\eqref{slN-dpToda} on the phase space $T_{\ve}$:
\begin{align*}
  &\Big\{ (I_{n}^t, I_{n}^{t+\frac{1}{M}}, \ldots, 
  I_n^{t+\frac{M-1}{M}}, V_n^t)_{n=1,  \ldots,N} \in K_{\ve}^{N(M+1)} 
  \\
  & \qquad \qquad ~\Big|~
    \val \bigl(\prod_{n=1}^N I_n^{t+\frac{k}{M}}\bigr) 
  < \val \bigl(\prod_{n=1}^N V_n^{t}\bigr) 
  ~(k=0,\ldots,M-1),
    \\
  & \qquad \quad \qquad
  \val \bigl(\prod_{n=1}^N I_n^{t+\frac{k}{M}}) ~(k=0,\ldots,M-1) 
  \text{ are distinct }
  \Big\}.
\end{align*}
We assume $\gcd(M,N) = 1$. 
Let $F_{\ve} \subset K_{\ve}[x,y]$ be the set of polynomials 
over $K_{\ve}$ defined by the similar formula to \eqref{slN-dpF}, \textit{i.e.}
\[
F_{\ve} = 
\Big\{ y^{M+1}+\sum_{j=0}^{M}{f_j(x)y^j}
  \in K_{\ve}[x,y] ~\Big|~
  \deg_x f_j \leq d_j, ~
  -f_1(x) \text{ is monic}
  \Big\}.
\]

The {\em tropicalization} of the above system becomes
the generalized ultradiscrete periodic Toda lattice 
$\mT(M,N)$, which is the piecewise-linear map: 
\begin{align}
    \begin{split}\label{ud-pToda}
    &Q_n^{t+1} = \min[W_n^t, Q_n^t - X_n^t], \\
    &W_n^{t+\frac{1}{M}} = Q_{n+1}^t + W_n^t - Q_n^{t+1},
    \end{split}
    \quad (n \in \Z/N \Z, ~t \in \Z/M),
    \\ \nonumber
    &\text{where }
    X_n^t = \min_{k=0,\ldots,N-1}[\sum_{j=1}^k (W_{n-j}^t-Q_{n-j}^t)],
\end{align}
on the phase space $\mT$:
\begin{align*}
  \mT &= 
  \Big\{(Q_{n}^t, Q_{n}^{t+\frac{1}{M}}, \ldots, Q_n^{t+\frac{M-1}{M}}, W_n^t)_{n=1,\ldots,N} \in \R^{(M+1)N}
  \\
  & \qquad \qquad
  ~\Big|~ \sum_n Q_n^{t+\frac{k}{M}} < \sum_n W_n^{t} ~ (k=0,\ldots,M-1), 
  \\
  & \qquad \quad \qquad
   \sum_n Q_n^{t+\frac{k}{M}} ~(k=0,\ldots,M-1) 
   \text{ are distinct } \Big\}.
\end{align*}
Here we set $\val(I_n^t) = Q_n^t$ and $\val(V_n^t) = W_n^t$.
The tropicalization of $F_{\ve}$ becomes the space of tropical polynomials
on $\mT$:
\begin{align}\label{slM-udF}
  \begin{split}
  \mathcal{F} =
  &\Bigl\{
  \min \Bigl[ (M+1) Y, ~
  \min_{j=1,\ldots,M} 
  \bigl[jY + \min[d_j X + F_{j,d_j},\ldots, 
  \\
  & \hspace{2.5cm}
   X+F_{j,1},F_{j,0}]\bigr],
   ~ F_0 \Bigr] ~\Big|~ F_{1,d_1} = 0, ~F_{j,i},
  F_0 \in \R 
 \Bigr\}.
  \end{split}
\end{align}
We write $\Phi$ for the map $\mT \to \mF$.

\subsection{Spectral curves for $T(M,N)$ and good tropicalization}

We continuously assume $\gcd(M,N)=1$. 

\begin{proposition}\label{prop:MNToda-curve}
For a generic point $\tau \in T_{\ve}$, 
that is a point in a certain Zariski open subset of $T_{\ve}$, 
the spectral curve $\phi \circ \psi(\tau)$ has a good tropicalization.
\end{proposition}

To show this proposition, we use the following lemma:

\begin{lemma}\label{lemma:MN}
  Fix $l \in K_{\ve}[x,y]$ and set $h_t = y^{M+1} - x^N y + tl$,
  where $t\in K_{\ve}$.
  Then $C^0(h_t)$ is non-singular in $(K_{\ve}^\times)^2$ except for
  finitely many $t$. 
\end{lemma}
\begin{proof}  
  Fix $a,b \in \Z$ as $Ma-Nb=1$.
  (It is always possible since $\gcd(M,N)=1$.) 
  Then the map $\nu: (K_{\ve}^\times)^2 \to (K_{\ve}^\times)^2; 
  ~(x,y) \mapsto (u, v) = (x^N/y^M, x^a/y^b)$ is holomorphic.
  The push forward of $h_t/y^{M+1}$ by $\nu$ becomes  
  $$ 
    \tilde{h} := (h_t/y^{M+1})_\ast = (1-u+t\tilde{l}) 
    \quad (\tilde{l} \in K_{\ve}[u,v,u^{-1},v^{-1}]).
  $$ 
  By using the following claim, we see that $C^0(\tilde{h})$ is non-singular.
  \begin{claim}
  Fix $f,g \in K_{\ve}[u,v]$ such that $C^0(f)$ is non-singular,
  and $f$ and $g$ are coprime to each other.
  Define 
  $$
    U = \{ t \in K_{\ve} ~|~ C^0(f+tg) \text{ is singular }\} \subset K_{\ve}.
  $$
  Then $U$ is a finite algebraic set.
  \end{claim}
  Then the lemma follows. 
\end{proof}

\begin{proof}(Proposition \ref{prop:MNToda-curve})
\\
Recall the definition of good tropicalization (Definition \ref{def2.1}).

The part (1) follows from Proposition \ref{prop:integrability} immediately.

The part (2): 
For any $f_{\ve} \in F_{\ve}$, it can be easily checked that if 
two points
$P_1,P_2\in TV(f_{\ve})$ exist on a same 
edge of the tropical curve, then $f_{\ve}^{P_1}=f_{\ve}^{P_2}$.
This fact implies that the set $\{f_{\ve}^P\,\vert\, P\in TV(f_{\ve})\}$ is finite.
Therefore, the set
\begin{align*}
\bigtriangleup=\{f_{\ve}\in F_{\ve}\,\vert\, C^{0}(f^P_{\ve}) 
&\mbox{ is non-reduced or singular in $(K_{\ve}^\times)^2$} 
\\
& \hspace{3cm}\mbox{ for some } P\in TV(f_{\ve}) \}
\end{align*}
is a union of finitely many
non-trivial algebraic subsets of $F_{\ve}\simeq K_{\ve}^{n_F}$.
Using Proposition \ref{prop:integrability} 
(with the map $\iota$ with any $\bar{\ve} \in \R_{>0}$)
and Lemma \ref{lemma:MN},
we conclude that $(\phi\circ\psi)^{-1}(\bigtriangleup)\subset T_{\ve}$
is an analytic subset with positive codimension.
(We need Lemma \ref{lemma:MN} when $f_{\ve}^P$ includes $y^{M+1} - x^N y$.) 
\end{proof}

\section{General solutions to $\mathcal{T}(M,N)$}
\subsection{Bilinear equation}

In the following we use a notation $[t] = t ~{\rm mod}~ 1$
for $t \in \frac{\Z}{M}$.
The following proposition gives the bilinear form for 
$\mT(M,N)$:
\begin{proposition}
Let $\{T_n^t \}_{n \in \Z; t \in \frac{\Z}{M}}$ be a set of 
functions with a quasi-periodicity,
$T_{n+N}^t = T_n^t + (an + bt +c)$
for some $a,b,c \in \R$. 
Fix $\delta^{[t]}, \theta^{[t]} \in \R$ such that 
$$ 
  (a) \text{ $\delta^{[t]} + \theta^{[t]}$ does not depend on $t$}, 
  \qquad 
  (b) ~2b-a < N \theta^{[t]} \text{ for } t \in \Z/M.
$$ 
Assume $T_n^t$ satisfies
\begin{align}\label{tau-function}
  T_n^t + T_n^{t+1+\frac{1}{M}} =
  \min[ T_n^{t+1} + T_n^{t+\frac{1}{M}}, 
       T_{n-1}^{t+1+\frac{1}{M}} + T_{n+1}^t + \theta^{[t]}].
\end{align}
Then $T_n^t$ gives a solution to \eqref{ud-pToda}
via the transformation:
\begin{align}\label{QW-T}
  \begin{split}
  &Q_n^t = T_{n-1}^{t} + T_{n}^{t+\frac{1}{M}}
           - T_{n-1}^{t+\frac{1}{M}} - T_{n}^{t}+\delta^{[t]},
  \\
  &W_n^t = T_{n-1}^{t+1} + T_{n+1}^t-
          T_{n}^{t} - T_{n}^{t+1} + \delta^{[t]}+\theta^{[t]}.
  \end{split}
\end{align}
\end{proposition}
We omit the proof since it is essentially same as that of $M=1$ case
in \cite[\S 3]{IT09}.

\begin{remark}
Via \eqref{QW-T}, \eqref{ud-pToda} is directly related to 
  \begin{align*}
    &T_n^t + T_n^{t+1+\frac{1}{M}} =
    T_n^{t+1} + T_n^{t+\frac{1}{M}} + X_{n+1}^t,
    \\ \nonumber
    &X_n^t = \min_{j=0,\ldots,M-1} 
      \bigl[
      j \theta^{[t]} 
      + T_{n-j-1}^{t+\frac{1}{M}}+T_{n-j-1}^{t+1}+T_n^t+T_{n-1}^t
      \\ \nonumber  & \hspace{3cm}  
      -(T_{n-1}^{t+1}+T_{n-1}^{t+\frac{1}{M}}+T_{n-j}^t+T_{n-j-1}^t) \bigr].
  \end{align*}
This is shown to be equivalent to \eqref{tau-function}
under the quasi-periodicity of $T_n^t$.
See \cite[Proposition 3.3 and 3.4]{IT09} for the proof. 
\end{remark}

\subsection{Example: $\mT(3,2)$}\label{sec:32}

We demonstrate a general solution to $\mT(3,2)$.
Take a generic point $\tau \in T_{\ve}$, 
and the spectral curve $C(f_{\ve})$ for $T(3,2)$ on $K_{\ve}$ 
is given by the zero of $f_{\ve} = \phi \circ \psi (\tau) \in F_{\ve}$:
$$
  f_{\ve} = y^4 + y^3 f_{30} + y^2(x f_{21} + f_{20}) + y(-x^2+x f_{11}+f_{10})
  + f_0.
$$
Due to Proposition \ref{prop:MNToda-curve}, 
$C(f_{\ve})$ has a good tropicalization.
The tropical curve $\Gamma :=TV(f_{\ve})$ in $\R^2$ 
is the indifferentiable points of $\xi := \Val(X,Y;f_{\ve})$: 
$$
  \min \big[ 4 Y, 3Y + F_{30}, 2 Y + \min[X+F_{21},F_{20}],
             Y + \min[2X,X+F_{11},F_{10}], F_0 \bigr].
$$

We assume that $\Gamma$ is smooth, then its genus is $g=2$.
See Figure \ref{GammaC} for $\Gamma$,
where we set the basis $\gamma_1, \gamma_2$ of $\pi_1(\Gamma)$.
\begin{figure}
\begin{center}
\unitlength=1.2mm
\begin{picture}(80,80)(0,0)

\put(0,5){\line(1,0){80}}

\thicklines

\put(0,0){\line(1,1){5}}
\put(5,5){\line(1,1){15}}
\put(5,5){\line(4,1){20}}
\put(25,10){\line(2,1){10}}
\put(20,20){\line(1,0){15}}
\put(35,20){\line(2,1){15}}
\put(20,60){\line(-1,1){20}}
\put(20,60){\line(2,-1){30}}
\put(50,45){\line(1,0){30}}
\put(50,27.5){\line(1,0){30}}
\put(35,15){\line(1,0){30}}
\put(25,10){\line(1,0){30}}

\put(35,15){\line(0,1){5}}
\put(20,20){\line(0,1){40}}
\put(50,45){\line(0,-1){17.5}}

\put(25,10){\circle*{1.5}} \put(23,6){$A_1$}
\put(35,15){\circle*{1.5}} \put(33,11){$A_2$}
\put(50,27.5){\circle*{1.5}} \put(48,23.5){$A_3$}
\put(50,45){\circle*{1.5}} \put(49,47){$R$}
\put(20,60){\circle*{1.5}} \put(19,62){$P$}
\put(5,5){\circle*{1.5}} \put(6,2){$Q$}

\thinlines

\put(35,36){\oval(20,20)}
\put(33,25.2){$>$}
\put(33,28){$\gamma_1$}
\put(23,15){\oval(10,6)}
\put(21,11.2){$>$}
\put(21,14){$\gamma_2$}

\end{picture}
\caption{Tropical spectral curve $\Gamma$ for $\mT(3,2)$}\label{GammaC}
\end{center}
\end{figure}
The period matrix $B$ for $\Gamma$ becomes
$$
  B = \begin{pmatrix}
        2 F_0 - 7 F_{11} + F_{20} & F_{11} - F_{20} \\
        F_{11} - F_{20} & F_{11} + F_{20}
      \end{pmatrix},
$$
and the tropical Jacobi variety $J(\Gamma)$ of $\Gamma$ is 
$$
  J(\Gamma) = \R^2 / \Z^2 B.
$$

We fix $6$ points on the universal covering space of $\Gamma$
as follows:
\begin{align*}
  &\vec{L} = \int_P^Q = (F_0-3F_{11},F_{11}),
  \\
  &\vec{\lambda}_1 = \int_Q^{A_3} = (F_{10}-2F_{11},-F_{20}),
  \quad 
  \vec{\lambda}_2 = \int_Q^{A_2} = (0, F_{20}-F_{30}),
  \\
  &\vec{\lambda}_3 = \int_Q^{A_1} = (0, F_{30}),
  \quad
  \vec{\lambda} = \int_R^P = (F_{10}-2F_{11},0).
\end{align*}
Here the path $\gamma_{Q \to A_3}$ from $Q$ to $A_3$ is chosen 
as $\gamma_{Q \to A_3} \cap \gamma_1 \cap \gamma_2 \neq \emptyset$.
Remark that 
$\vec{\lambda} = \vec{\lambda}_1 + \vec{\lambda}_2 + \vec{\lambda}_3$
holds.

\begin{proposition}\label{prop:T32}
  Fix $\bZ_0 \in \R^2$.
  The tropical theta function $\Theta(\bZ;B)$ satisfies
  the following identities:
  \begin{align}\label{T32Fay}
    \begin{split}
    &\Theta(\bZ_0) + \Theta(\bZ_0+\vec{\lambda}+\vec{\lambda}_i)
    \\ &\quad 
     = \min [ \Theta(\bZ_0+\vec{\lambda}) + \Theta(\bZ_0+\vec{\lambda}_i),
              \Theta(\bZ_0-\vec{L}) 
              + \Theta(\bZ_0+\vec{L}+\vec{\lambda}+\vec{\lambda}_i)
              + \theta_i],
    \end{split}
  \end{align}
  for $i=1,2,3$, where $\theta_1 = F_0-3 F_{11}$ and 
  $\theta_2 = \theta_3 = F_0-2 F_{11}$. 
\end{proposition}
\begin{proof}
  Since the curve $C(f_{\ve})$ is hyperelliptic, 
  we fix a non-singular odd theta characteristic as
  $(\alpha,\beta)=((\frac{1}{2},0),(\frac{1}{2},\frac{1}{2}))$
  following Lemma \ref{lemma:theta-ch}.
  By setting $(P_1,P_2,P_3,P_4) = (R,Q,P,A_{4-i})$ in Theorem \ref{tropicalFay}
  for $i=1,2,3$, we obtain \eqref{T32Fay}. 
\end{proof}

Now it is easy to show the following:
\begin{proposition}
  (i) Fix $\bZ_0 \in \R^2$ and $\{i,j\} \subset \{1,2,3\}$, 
  and define $T_n^t$ by
  \begin{align*}
    \begin{split}
    &T_n^t = \Theta(\bZ_0-\vec{L}n+\vec{\lambda}t),
    \\
    &T_n^{t+\frac{1}{3}} 
     = \Theta(\bZ_0-\vec{L}n+\vec{\lambda}t+\vec{\lambda}_i),
    \\
    &T_n^{t+\frac{2}{3}} 
     = \Theta(\bZ_0-\vec{L}n+\vec{\lambda}t+\vec{\lambda}_i+\vec{\lambda}_j),
    \end{split}
    \quad (t \in \Z).
  \end{align*}
  Then they satisfy the bilinear equation \eqref{tau-function}
  with 
  $\theta^{[0]} = \theta_i$, $\theta^{[\frac{1}{3}]} = \theta_j$ and 
  $\theta^{[\frac{2}{3}]} = \theta_k$,
  where $\{k\} = \{1,2,3\} \setminus \{i,j\}$. 
  \\
  (ii) With (i) and 
  $\delta^{[\frac{k}{3}]}=F_0-2F_{11}-\theta^{[\frac{k}{3}]} ~(k=0,1,2)$, 
  we obtain a general solution to $\mT(3,2)$. 
\end{proposition}

\begin{remark}
  Depending on a choice of $\{i,j\} \subset \{1,2,3\}$,
  we have $3!=6$ types of solutions. 
  This suggests a claim for the isolevel set $\Phi^{-1}(\xi)$:
  $$
    \Phi^{-1}(\xi) \simeq J(\Gamma)^{\oplus 6}.
  $$
\end{remark}

\subsection{Conjectures on $\mT(M,N)$}

We assume $\gcd(M,N)=1$ again. 
Let $\Gamma$ be the smooth tropical curve given by 
the indifferentiable points of a tropical polynomial $\xi \in \mF$ 
\eqref{slM-udF}.
We fix the basis of $\pi_1(\Gamma)$ by using 
$\gamma_{i,j} ~(i=1,\ldots,M, ~j=1,\ldots,d_i)$ as 
Figure \ref{fig:GammaMN}.
The genus $g = \frac{1}{2}(N-1)(M+1)$ of $\Gamma$ can be obtained 
by summing up $d_j$ from $j=1$ to $\max_{j=1.\ldots,M} \{j ~|~ d_j\geq 1\}$.

\begin{figure}
\begin{center}
\unitlength=1.2mm
\begin{picture}(80,80)(0,0)

\put(0,5){\line(1,0){80}}

\thicklines

\put(0,0){\line(1,1){5}}
\put(5,5){\line(1,1){10}}
\put(15,15){\line(5,1){5}}
\put(15,15){\line(1,2){5}}
\put(20,25){\line(0,1){4}}
\put(20,25){\line(4,1){4}}

\put(5,5){\line(5,1){10}}
\put(25,10){\line(-4,-1){5}}
\put(25,10){\line(2,1){10}}
\put(35,20){\line(-1,0){5}}
\put(35,20){\line(2,1){5}}
\put(20,60){\line(-1,1){20}}
\put(20,60){\line(2,-1){5}}
\put(60,45){\line(1,0){20}}
\put(60,45){\line(-4,1){16}}
\put(44,49){\line(-3,1){9}}
\put(50,27.5){\line(1,0){30}}
\put(50,27.5){\line(-2,-1){3}}
\put(35,15){\line(1,0){30}}
\put(25,10){\line(1,0){30}}
\put(44,49){\line(0,-1){5}}

\put(35,15){\line(0,1){5}}
\put(20,60){\line(0,-1){20}}
\put(20,43){\line(1,0){4}}
\put(50,27.5){\line(0,1){5}}
\put(60,38){\line(0,1){7}}
\put(60,38){\line(-2,-1){4}}
\put(60,38){\line(1,0){20}}

\put(25,10){\circle*{1.5}} \put(23,6){$A_1$}
\put(35,15){\circle*{1.5}} \put(33,11){$A_2$}
\put(50,27.5){\circle*{1.5}} \put(48,23.5){$A_i$}
\put(20,60){\circle*{1.5}} \put(19,62){$P$}
\put(5,5){\circle*{1.5}} \put(6,2){$Q$}
\put(60,38){\circle*{1.5}} \put(58,34){$A_M$}
\put(60,45){\circle*{1.5}} \put(59,47){$R$}

\put(44,35){$\vdots$}
\put(23,33){$\vdots$}
\put(38,25){$\ldots$}
\put(27,22){$\ldots$}
\put(20,12){$\ldots$}

\thinlines

\put(54,41){\circle{7}}
\put(52.5,36.8){$>$}  \put(51,40){$\gamma_{1,d_1}$}
\put(25,50){\circle{7}}
\put(23.5,45.8){$>$}  \put(23,49){$\gamma_{1,1}$}

\put(22,20.5){\circle{5}}
\put(20.5,17.4){$>$}  \put(8,20){$\gamma_{M-1,1}$}
\put(31,16.5){\circle{5}}
\put(29.5,13.4){$>$}  \put(36,17){$\gamma_{M,d_M}$}
\put(15,10.5){\circle{5}}
\put(13.5,7.4){$>$}  \put(5,12){$\gamma_{M,1}$}

\end{picture}
\caption{Tropical spectral curve $\Gamma$ for $\mT(M,N)$}
\label{fig:GammaMN}
\end{center}
\end{figure}

Fix three points $P$, $Q$, $R$ on the universal covering space 
$\tilde{\Gamma}$ of $\Gamma$ as Figure \ref{fig:GammaMN}, and define
\begin{align*}
  \vec{L} = \int_P^Q, \qquad 
  \vec{\lambda} = \int_R^P.
\end{align*}
Fix $A_i ~(i=1,\ldots,M)$ on $\tilde{\Gamma}$ as Figure \ref{fig:GammaMN},
such that 
\begin{align*}
  \vec{\lambda}_i = \int_Q^{A_{M+1-i}} ~ (i=1,\ldots,M)
\end{align*}
satisfy $\vec{\lambda} = \sum_{i=1}^M \vec{\lambda}_i$.

We expect that
the bilinear form \eqref{tau-function} is obtained as a consequence of
the tropical Fay's identity \eqref{trop-fay},
by setting $(P_1,P_2,P_3,P_4) = (R,Q,P,A_i)$ in Theorem \ref{tropicalFay}.
The followings are our conjectures:
\begin{conjecture}\label{conj:1}
Let $\mathcal{S}_M$ be the symmetric group of order $M$.
Fix $\bZ_0 \in \R^g$ and $\sigma \in \mathcal{S}_M$, and set
$$
  T_n^{t+\frac{k}{M}} = \Theta(\bZ_0 - \vec{L}n + \vec{\lambda} t + 
                          \sum_{i=1}^{k} \lambda_{\sigma(i)})
$$
for $k=0,\ldots,M-1$.
Then the followings are satisfied:
\\
(i) $T_n^t$ satisfy \eqref{tau-function} with some $\theta^{[t]}$.
\\
(ii) $T_n^t$ gives a general solution to $\mT(M,N)$ via \eqref{QW-T}. 
\end{conjecture}

\begin{conjecture}\label{conj:2}
The above solution induces the isomorphism map 
from $J(\Gamma)^{\oplus M!}$  to the isolevel set $\Phi^{-1}(\xi)$.  
\end{conjecture}

\begin{remark}
  In the case of $\mT(1,g+1)$ and $\mT(2g-1,2)$,
  the smooth tropical spectral curve $\Gamma$ is 
  hyperelliptic and has genus $g$.
  For $\mT(1,g+1)$, Conjectures \ref{conj:1} and \ref{conj:2} 
  are completely proved \cite{IT09,IT09b}.  
  For $\mT(3,2)$, Conjecture \ref{conj:1} is shown in \S \ref{sec:32}.  
\end{remark}

\section*{Acknowledgements}

R.~I. thanks the organizers of the international conference 
``Infinite Analysis 09 
--- New Trends in Quantum Integrable Systems'' 
held in Kyoto University in July 2009,
for giving her an opportunity to give a talk.
She also thanks Takao Yamazaki for advice on the manuscript.

R.~I. is partially supported by Grant-in-Aid for Young Scientists (B)
(19740231).

S.~I. is supported by KAKENHI
(21-7090).

\end{document}